%% file: svdForComplexes180520A.tex
\documentclass[12pt, leqno]{article} 
\usepackage{amsmath,amscd,amsthm,amssymb,amsxtra,latexsym,epsfig,epic,graphics}
\usepackage[matrix,arrow,curve]{xy}
\usepackage{graphicx}
\usepackage{diagrams}
\usepackage{hyperref}
\usepackage{enumerate}
\usepackage[alphabetic,lite]{amsrefs} 

\input preamble.tex

\def\PP{{\mathbb P}}
\def\RR{{\mathbb R}}

\def\QQ{{\mathbb Q}}
\def\FF{{\mathbb F}}

\def\image{{\rm image\,}}

\def\bR{{\mathbb R}}

\makeatletter
\def\Ddots{\mathinner{\mkern1mu\raise\p@
\vbox{\kern7\p@\hbox{.}}\mkern2mu
\raise4\p@\hbox{.}\mkern2mu\raise7\p@\hbox{.}\mkern1mu}}
\makeatother

\makeatletter
\def\Ddots{\mathinner{\mkern1mu\raise\p@
\vbox{\kern7\p@\hbox{.}}\mkern2mu
\raise4\p@\hbox{.}\mkern2mu\raise7\p@\hbox{.}\mkern1mu}}
\makeatother

\usepackage{marginnote}

\usepackage{times}
\newdimen\x \x=12pt

\usepackage{color}

\date{May 19, 2018}
\title{Singular value decomposition of complexes}
\author{Danielle A. Brake, Jonathan D. Hauenstein, Frank-Olaf Schreyer,  \\ Andrew J. Sommese, and Michael E. Stillman}

\begin{document}

\maketitle

\begin{abstract}
\noindent Singular value decompositions of matrices
are widely used in numerical linear algebra
with many applications.  In this paper,
we extend the notion of singular value decompositions
to finite complexes of real vector spaces.
We provide two methods to compute them
and present several applications.
\end{abstract}

\section{Introduction}

For a matrix $A\in\bR^{m\times k}$, a singular
value decomposition (SVD) of $A$ is
$$A = U\Sigma V^t$$
where $U\in\bR^{m\times m}$ and $V\in\bR^{k\times k}$
are orthogonal and $\Sigma\in\bR^{m\times k}$
is diagonal with nonnegative real numbers on the diagonal.  The diagonal entries of $\Sigma$,
say $\sigma_1 \geq \cdots \geq \sigma_{\min\{m,k\}} \geq 0$
are called the singular values of $A$
and the number of nonzero singular values is equal
to the rank of $A$.
Many problems in numerical linear algebra
can be solved using a singular value decomposition
such as pseudoinversion,
least squares solving,
and low-rank matrix approximation.

A matrix $A\in\bR^{m\times k}$ defines
a linear map $A:\bR^k\rightarrow\bR^m$
via $x\mapsto Ax$ denoted
$$\bR^m \lTo^{A} \bR^k.$$
Hence, matrix multiplication simply
corresponds to function composition.
For example, if $B\in\bR^{\ell\times m}$, then
$B\circ A:\bR^k\rightarrow\bR^\ell$ is defined
by $x\mapsto BAx$ denoted
$$\bR^\ell \lTo^{B} \bR^m \lTo^{A} \bR^k.$$
If $B\circ A = 0$, then this composition
forms a {\em complex} denoted
$$0 \lTo \bR^\ell \lTo^{B} \bR^m \lTo^{A} \bR^k \lTo 0.$$
In general, a finite complex of finite dimensional
$\RR$-vector spaces
$$
0 \lTo C_0 \lTo^{A_1} C_{1} \lTo^{A_{2}} \ldots \lTo^{A_{n-1}} C_{n-1} \lTo^{A_n} C_n \lTo 0
$$
consists of vector spaces  $C_i \cong \RR^{c_i}$
and differentials given by matrices $A_{i}$
so that $A_{i}\circ A_{i+1} = 0$.
We denote such a complex by~$C_\bullet$
and its $i^{\rm th}$ homology group as
$$
H_i=H_i(C_\bullet) = \frac{ \ker A_i}{\image A_{i+1}}
$$
with $h_i = \dim H_i$.  Complexes are standard tools that occur in many
areas of mathematics including differential equations, e.g., \cite{AFW06,AFW10}.
One of the reasons for developing a singular
value decomposition of complexes
is to compute the dimensions $h_i$
efficiently and robustly via numerical methods
when each $A_i$ is only known approximately, say $B_i$.
For example, if we know that $\rank A_i = r_i$,
then $h_i$ could easily be computed via
$$ c_i=r_{i}+r_{i+1}+h_i.$$
One option would be to compute
the singular value decomposition of each
$B_i$ in order to compute the rank $r_i$
of $A_i$ since the singular value decomposition
is an excellent rank-revealing numerical method.
However, simply decomposing each~$B_i$
ignores the important information that the
underlying matrices $A_i$ form~a~complex.

The key point of this paper is that
we can utilize information about the
complex to provide more specific information
that reflects the structure it imposes.

\begin{theorem}[Singular value decomposition of complexes]\label{main thm} Let $A_1,\ldots,A_n$ be a sequence of matrices $A_i \in \RR^{c_{i-1}\times c_i}$ which define a complex $C_\bullet$, i.e.
\mbox{$A_{i}\circ A_{i+1} = 0$}.
Let $r_i = \rank A_i$ and $h_i = c_i - (r_i + r_{i+1})$.
Then, there exists sequences $U_0,\ldots,U_n$
and $\Sigma_1,\dots,\Sigma_n$ of orthogonal
and diagonal matrices, respectively, such that
$$
U_{i-1}^t\circ A_i \circ U_i = \bordermatrix{
& r_{i} & r_{i+1} & h_i \cr
r_{i-1} & 0 & 0& 0 \cr
r_{i} & \Sigma_i  & 0 & 0 \cr
h_{i-1} & 0 & 0 & 0 \cr
} := \overline{\Sigma}_i
$$
where all diagonal entries of $\Sigma_i$ are
strictly positive.
Moreover, if every $r_i >0$ and at least one
$h_i >0$, then the
orthogonal matrices $U_i$
can be chosen such that $\det U_i = 1$,
i.e., each $U_i$ is a special orthogonal matrix.
\end{theorem}

\noindent
The diagonal entries of $\Sigma_1,\dots,\Sigma_n$
are the {\em singular values of the complex}.

We develop
two methods that
utilize the complex structure
to compute a singular
value decomposition of $C_\bullet$.
The successive projection method
described in Algorithm \ref{SuccessiveProjectionMethod} uses the orthogonal projection
$$P_{i-1}\colon C_{i-1} \to \ker A_{i-1} $$
together with the singular value decomposition of $P_{i-1} \circ A_i$.
The second method, described in Algorithm \ref{LaplacianMethod}, is based on the
{\it Laplacians}
$$
\Delta_i = A_i^t \circ A_i + A_{i+1}\circ A_{i+1}^t.
$$
Both methods can be applied to numerical
approximations $B_i$ of $A_i$.

Organization of this paper is as follows.  Section \ref{basics} proves Theorem \ref{main thm} and
collects a number of basic facts
along with defining the pseudoinverse of a complex.
Section \ref{algorithms} describes
the algorithms mentioned above
and illustrates them on an example.
Section \ref{applications} considers
projecting an arbitrary sequence of matrices
onto a complex.
Section~\ref{syzygies} provides
an application to computing
Betti numbers of minimal free resolutions of graded modules over the polynomial ring
$\QQ[x_0,\ldots,x_n]$ which combines our
method with ideas from \cite{EMSS}.

{\bf Acknowledgement.}
DAB and JDH was supported in part by NSF grant ACI-1460032.
JDH was also supported by
Sloan Research Fellowship BR2014-110 TR14.
AJS was supported in part by NSF ACI-1440607.
FOS is grateful to Notre Dame for its hospitality
when developing this project.  This work is a contribution to his Project 1.6 of the SFB-TRR 195 "Symbolic Tools in
Mathematics and their Application" of the German Research Foundation (DFG). MES was supported in part by NSF grant DMS-1502294 and is grateful to Saarland University for its hospitality during a month of intense work on this project.

\section{Basics}\label{basics}

We start with a proof of our main theorem.

{\it Proof of Theorem \ref{main thm}.}
For convenience, we set $A_0 = A_{n+1} = 0$
to compliment $A_1,\dots,A_n$
that describe the complex.
By the homomorphism theorem
$$
(\ker A_i)^\perp \cong \image A_i .
$$
The singular value decomposition
for a complex
follows by applying singular value decomposition
to this isomorphism and extending
an orthonormal basis of these spaces
to orthonormal basis of $\RR^{c_{i-1}}$ and $\RR^{c_i}$.
Since $\image A_{i+1} \subset \ker A_i $ we have an orthogonal direct sum
$$
(\ker A_i)^\perp \oplus \image A_{i+1} \subset \RR^{c_i}
$$
with
$$
H_i := ((\ker A_i)^\perp \oplus \image A_{i+1})^\perp =  \ker A_i \cap  \image A_{i+1}^\perp\cong \frac{ \ker A_i}{\image A_{i+1}}.
$$
With respect to these subspaces, we can decompose $A_i$ as
$$\bordermatrix{
& (\ker A_{i})^\perp&\image A_{i+1}  & H_i \cr
(\ker A_{i-1})^\perp& 0 & 0 & 0 \cr
 \image A_i & \Sigma_i & 0 & 0 \cr
H_{i-1} & 0 & 0 & 0 \cr
}.
$$
Indeed, $A_i$ has no component mapping to  $(\image A_i)^\perp$, which explains six of the zero blocks, and $\ker A_i =( \ker A_i)^{\perp \perp}= \image A_{i+1} \oplus H_i$ explains the remaining two.
Take $U_i$ to be the orthogonal matrix whose column vectors form the orthonormal basis of the spaces
$(\ker A_i)^\perp$ and  $\image A_{i+1}$ induced from the singular value decomposition of $ (\ker A_i)^\perp \to \image A_i$
and $(\ker A_{i+1})^\perp \to \image A_{i+1}$
extended by an orthogonal basis of $H_i$ in the decomposition
 $$(\ker A_i)^\perp \oplus \image A_{i+1} \oplus H_i= \RR^{c_i}.$$
The linear map $A_i$ has, in terms of these bases, the description $U_{i-1}^t\circ A_i \circ U_i$ which has the desired shape.

Finally, to achieve $\det U_i = 1$, we may, for $1 \le k \le r_i$, change signs of the~$k^{\rm th}$ column in $U_i$ and $(r_{i-1}+k)^{\rm th}$ column of $U_{i-1}$ without changing the result of the conjugation. If $h_i>0$, then changing the sign of any of the last $h_i$ columns of~$U_i$ does not affect the result either.
Thus, this gives us enough freedom to reach $\det U_i =1 $ for all $i=0, \ldots, n$.
 \qed

\begin{corollary}[Repetition of eigenvalues]\label{repetition}
Suppose that $A_1, \ldots A_n$ define a complex
with $A_0 = A_{n+1} = 0$.
Let $\Delta_i=A_i^t\circ A_i + A_{i+1} \circ A_{i+1}^t $ be the corresponding Laplacians.
Then, using the orthonormal bases described
by the $U_i$'s from Theorem \ref{main thm}, the Laplacians have the shape
$$
\Delta_i =\bordermatrix{
& r_{i} & r_{i+1} & h_i \cr
r_{i} & \Sigma_{i}^2 & 0 & 0 \cr
r_{i+1} & 0 &\Sigma_{i+1}^2  & 0 \cr
h_i & 0 & 0 & 0 \cr
}.
$$
In particular,
\begin{enumerate}
\item if $r_i=\rank A_i$ and $\sigma^i_1 \ge \sigma^i_2 \ge \ldots \sigma^i_{r_i} >0$ are
 the singular values of $A_i$, then each $(\sigma^i_k)^2$ is both an eigenvalue $\Delta_i$ and $\Delta_{i-1}$;
 \item $\ker \Delta_i = H_i$.
\end{enumerate}
\end{corollary}
\begin{proof} The structure of $\Delta_i$
follows immediately from the structure described
in Theorem~\ref{main thm}.
The remaining assertions are immediate consequences. \end{proof}

Let $A_i^+$ denote the Moore-Penrose
pseudoinverse of the $A_i$.
Thus, a singular value decomposition
$$
 A_i  = U_{i-1} \circ \begin{pmatrix}
 0 & 0 & 0 \cr
 \Sigma_i  & 0 & 0 \cr
 0 & 0 & 0 \cr
\end{pmatrix} \circ U_i^t
\,\,\,\,\,\,\mbox{yields}\,\,\,\,\,\,\,
 A_i^+ = U_{i} \circ \begin{pmatrix}
 0 &   \Sigma_i^{-1}  & 0 \cr
0 & 0 & 0 \cr
 0 & 0 & 0 \cr
\end{pmatrix} \circ U_{i-1}^t.
$$

\begin{proposition}\label{projection to homology}
Suppose that $A_1,\ldots, A_n$ define a complex
with $A_0 = A_{n+1} = 0$.
Then, $A_{i+1}^+ \circ A_i^+ =0$ and
$$ id_{\RR^{c_i}}- (A_i^+ \circ A_i+ A_{i+1}\circ A_{i+1}^+)$$
defines the orthogonal projection of $\RR^{c_i}$
onto the homology $H_i$.
\end{proposition}
\begin{proof} We know that
$A_i^+\circ A_i$ defines the projection onto $(\ker A_i)^\perp$ and $A_{i+1} \circ A_{i+1}^+$ defines
the projection onto $\image A_{i+1}$.
The result follows immediately since
these spaces are orthogonal and $H_i=((\ker A_i)^\perp \oplus \image A_{i+1} )^\perp$.\end{proof}

We call
$$
0 \rTo \RR^{c_0} \rTo^{A_1^+ } \RR^{c_1} \rTo^{A_2^+ } \ldots \rTo^{A_n^+ } \RR^{c_n} \rTo 0.
$$
the {\em pseudoinverse complex} of
$$
0 \lTo \RR^{c_0} \lTo^{A_1 } \RR^{c_1} \lTo^{A_2} \ldots \lTo^{A_n } \RR^{c_n} \lTo 0.
$$

\begin{remark} If the matrices $A_i$ have entries in  a subfield $K \subset \RR$, then the pseudoinverse complex is also defined over $K$.
This follows since
the pseudoinverse is uniquely determined by the Penrose relations \cite{Pen}:
\begin{align*}
&A_i\circ A_i^+ \circ A_i = A_i, &A_i\circ A_i^+=(A_i\circ A_i^+)^t,\cr
&A_i^+\circ A_i \circ A_i^+ = A_i^+, \;&A_i^+\circ A_i=(A_i^+ \circ A_i)^t,&
\end{align*}
which form an algebraic system of equations for the entries of $A_i^+$ with a unique solution
whose coefficients are in $K$.
In particular, this holds for $K=\QQ$.

If the entries of the matrices are in the
finite field $\FF_q$,
the pseudoinverse of $A_i$ is well defined over $\FF_q$  with respect to the dot-product on $\FF_q^{c_i}$ and $ \FF_q^{c_{i-1}}$ if $$\ker A_i \cap (\ker A_i)^\perp = 0 \subset \FF_q^{c_i} \hbox{ and }\image A_i \cap (\image A_i)^\perp =0 \subset \FF_q^{c_{i-1}}.$$

We have implemented the computation of the pseudoinverse complex for double precision floating-point numbers
$\RR_{53}$, the rationals
$\QQ$, and finite fields $\FF_q$ in our Macaulay2 package  \href{https://www.math.uni-sb.de/ag/schreyer/index.php/computeralgebra}{SVDComplexes}.
\end{remark}

\section{Algorithms}\label{algorithms}

We present two algorithms for computing
a singular value decomposition
of a complex followed by an example.

\begin{algorithm}[Successive projection method]\label{SuccessiveProjectionMethod} $\quad$\\
INPUT: A sequences $B_1,\ldots,B_n$  of floating point matrices which are approximations of a complex $A_1,\ldots,A_n$;
a threshold $b$ for which we took $b=10^{-4}$ as default value in our implementation. \\
OUTPUT: Integers  $r_1,\ldots r_n$ and floating point approximations $U_0, \ldots, U_n$ of orthogonal matrices which approximate the singular value decomposition of the corresponding complex.\\

\begin{enumerate}[1.]
\item Set $r_0=0$, $Q_0=0$ and $P_0 = \id_{C_0}$.
\item For $i$ from $1$ to $n$ do
\begin{enumerate}[a.]
\item Compute the $(c_{i-1}-r_{i-1}) \times c_i$ matrix $\widetilde B_i= P_{i-1}\circ B_i$.
\item Compute the  singular value decomposition of $\widetilde B_i$, i.e. the diagonal matrix $\widetilde \Sigma_i$ of the singular values $\sigma^i_1\ge \sigma^i_2\ge \ldots$
and orthogonal matrices $\widetilde U_{i-1}, \widetilde V_i^t$ such that
$$ \widetilde B_i=\widetilde U_{i-1}\circ \widetilde \Sigma_i \circ \widetilde V_i^t.$$
\item Decide how many singular values of $\widetilde \Sigma_i$ are truly non-zero, e.g.
for $n_i =\min\{c_{i-1}-r_{i-1},c_i\}$ and for a threshold $b$, say $b=10^{-6}$, take
$$r_i= \begin{cases} \min \{j<n_i  \mid b\sigma^i_j\ge \sigma^i_{j+1}\}, &\hbox{ if this set is non-empty} \cr
n_i & \hbox{ else} \end{cases}$$
\item  Decompose
$$\widetilde V_i^t=\begin{pmatrix}Q_i  \cr  P_i \end{pmatrix}$$
 into submatrices consisting of the first $r_i$  and last $c_i-r_i$ rows of $\widetilde V_i^t$.
 (So $P_i$ defines an approximation of the orthogonal projection $C_i \to \ker A_i$ if our guess for $r_i$ was correct.)
\item Define $$U^t_{i-1}=\begin{pmatrix}Q_{i-1}   \cr
\widetilde U^t_{i-1}\circ P_{i-1} \end{pmatrix}.$$
\item If $i=n$ then $U_n=\widetilde V_n^t$.
\end{enumerate}
\item Return  $U_0,\ldots,U_n$ and $r_1,\ldots,r_n$.
\end{enumerate}
\end{algorithm}

{\it Proof} of concept. We will show that the algorithm gives a good approximation, provided that
\begin{enumerate}[i)]
\item the approximation $B_i$ of $A_i$ is sufficiently good,
\item we make the correct decisions in step 2.c  and
\item we compute with high enough precision.
\end{enumerate}
By induction on $i$ we will see that $P_i$ defines an approximation of the orthogonal projection $C_i \to \ker A_i$. Since $V_i^t$ is approximately orthogonal
$$
\begin{pmatrix}Q_{i}   \cr P_{i} \end{pmatrix}\circ \begin{pmatrix}Q^t_{i}   & P^t_{i} \end{pmatrix} \approx \begin{pmatrix}\id_{r_i} & 0 \cr   0 & \id_{c_i-r_i}\end{pmatrix}
$$
where $\id_k$ denotes a $k\times k$ identity matrix,
we  additionally  conclude that
$Q_i$ is
an approximation
 of  the orthogonal projection $C_i \to (\ker A_i)^\perp$.
 This is trivially true in case $i=0$, since $A_0=0$.

 For the induction step from $i-1$ to $i$, we note that $B_i \approx A_i$ and $\image A_i \subset \ker A_{i-1}$ implies that $Q_{i-1}\circ B_i \approx 0$.
 So $A_i$ and $P_{i-1} \circ B_i=\widetilde B_i$ have the same `large' singular values. From
 $$
 \widetilde B_i = \widetilde U_i \circ \widetilde \Sigma_i \circ V_i^t
 $$
 and
 $$
 V_i^t = \begin{pmatrix}Q_{i}   \cr P_{i} \end{pmatrix}
 $$
we conclude the assertion that $P_i$ defines approximately the orthogonal projection $C_i \to \ker A_i$  under the assumption, that  our choice of $r_i$ is correct. Moreover,
\begin{align*}
U^t_{i-1} \circ A_i \circ U_i &\approx U_{i-1}^t\circ B_i \circ U_i \cr
&=
\begin{pmatrix}Q_{i-1} \cr   \widetilde U_{i-1}^t \circ P_{i-1}\end{pmatrix} \circ B_i \circ \begin{pmatrix}Q_{i}^t  &  P_{i}^t \circ \widetilde U_i \end{pmatrix} \cr
& \approx
\begin{pmatrix}0 \cr   \widetilde U_{i-1}^t \circ \widetilde B_i\end{pmatrix} \circ \begin{pmatrix}Q_{i}^t  &  P_{i}^t \circ \widetilde U_i \end{pmatrix} \cr
&\approx \begin{pmatrix}  0 \cr \widetilde U_{i-1}^t \circ \widetilde U_{i-1} \circ \widetilde \Sigma_i \circ \begin{pmatrix}Q_{i}   \cr P_{i} \end{pmatrix}
\end{pmatrix} \circ \begin{pmatrix}Q_{i}^t  &  P_{i}^t \circ \widetilde U_i \end{pmatrix} \cr
& \approx \begin{pmatrix} 0 \cr
 \widetilde \Sigma_i \circ \begin{pmatrix} \id_{r_i} & 0 \cr   0 & \id_{c_i-r_i}\widetilde U_i \end{pmatrix}
\end{pmatrix} \cr
& \approx \begin{pmatrix} 0 & 0 \cr  \Sigma_i & 0 \cr  0 & 0 \cr \end{pmatrix} \cr
\end{align*}
since
$$
\widetilde \Sigma_i \circ \begin{pmatrix} 0 \cr   \id_{c_i-r_i}\end{pmatrix}  \approx 0.
$$
This shows that the desired approximation holds.
\qed

\begin{remark} To get more confidence in the correctness of the computation of $r_1,\ldots,r_n$ we can alter step 2.c. A natural approach is to start with two
approximations $B_1,\ldots, B_n$ and $B'_1,\ldots, B'_n$ in different precisions, and to determine $r_1,\ldots r_n$ as the number of stable singular values, i.e. the singular values which have approximately the same value in both computations.
\end{remark}

Our second method quite frequently does not need approximation in two different precisions. It is based on computing  with the {\it Laplacians}
$$
\Delta_i = A_i^t \circ A_i + A_{i+1}\circ A_{i+1}^t.
$$
Note that $\ker \Delta_i \cong H_i$.

\begin{algorithm}[Laplacian method]\label{LaplacianMethod}
$\quad$ \\
INPUT: A sequence $B_1,\ldots,B_n$ of floating points approximations of a complex $A_1,\ldots, A_n$, whose Laplacians have no multiple eigenvalues; a threshold $b$ for the relative precision for equality of eigenvalues. In our implementation we took $b=10^{-4}$ as the default value.\\
OUTPUT: Integers  $r_1,\ldots r_n$ and floating point approximations $U_0, \ldots, U_n$ of orthogonal matrices which approximate the singular value decomposition of the corresponding complex.\\
\begin{enumerate}
\item Compute diagonalisations $D_i$ of the symmetric semi-positive matrices
$$
\Delta'_i=B_i^t \circ B_i + B_{i+1}\circ B_{i+1}^t
$$ and orthogonal matrices $U'_i \in SO(c_i)$ such that
$$
\Delta'_i =U'_i\circ D_i \circ {U'}^t_i
$$
\item If some $D_i$ has a non-zero eigenvalue with higher multiplicity abort.
\item Let $r_i$ be the number of eigenvalues values which occur up to a chosen relative precision both in $D_{i-1}$ and $D_i$.
\item  Compute the corresponding  $c_i\times c_i$ permutation matrices $P_i$, which put the $r_i+r_{i+1}$ common
diagonal entries of $D_i$ into the first positions and set $U_i''= U'_i \circ P_i$.
\item Compute
$$U_{i-1}''^t \circ B_i \circ U_i'' \approx \begin{pmatrix} 0 & 0 & 0 \cr
\Sigma_i' & 0 & 0 \cr
0 & 0 & 0 \cr \end{pmatrix}
$$
where the $r_i \times r_i$ matrix $\Sigma_i'$ is approximately a diagonal matrix.
\item Let $\Sigma_i$ be the diagonal matrix whose entries are the absolute values of the diagonal entries of $\Sigma_i'$.
\item Inductively, for $i$ from $1$ to $n$ change the signs of the eigenvectors given by the columns of $U_i''$ to obtain orthogonal matrices $U_i$ such that
$$
U_{i-1}^t \circ B_i \circ U_i \approx \begin{pmatrix} 0 & 0 & 0 \cr
\Sigma_i & 0 & 0 \cr
0 & 0 & 0 \cr \end{pmatrix} =: \overline \Sigma_i.
$$
\item Return $U_0,\ldots,U_n$ and $r_1,\ldots,r_n$.
\end{enumerate}
\end{algorithm}

{\it Proof} of concept. We show that we the algorithm produces a good approximation in case
\begin{enumerate}[i)]
\item the Laplacians $\Delta_i =A_i^t \circ A_i + A_{i+1}\circ A_{i+1}^t$ have no non-zero multiple eigenvalues,
\item the approximations $B_i$ of the $A_i$ are good enough; in particular,  the disturbed non-zero eigenvalues stay apart,
\item the disturbed zero eigenvalues do not accidentally coincide up to a large relative precision, and
\item we compute with high enough precision.
\end{enumerate}
Indeed, by ii) and iii) we determine the ranks in step 3 correctly. So in step 5 we will reach approximately a diagonal matrix,
and it remains to adjust the signs. \qed

\begin{remark}  For the choice of the thresholds $b$ in Algorithm \ref{SuccessiveProjectionMethod} and Algorithm \ref{LaplacianMethod} we have only experimental
evidence. In particular our  default value $10^{-4}$ has no justification, not even heuristically. We leave it as an open problem to derive a justified choice, which might depend also on the ranks $c_i =\rank C_i$ of the $\RR$-vector spaces in the complex.
\end{remark}

\begin{example}\label{exampleOfSection2} We consider the complex
$$
0 \lTo \RR^{3} \lTo^{A_1 } \RR^{5} \lTo^{A_2} \RR^{5} \lTo^{A_3 } \RR^{3} \lTo 0
$$
defined by the matrices
{\small
$$
\begin{pmatrix}14&
      {-4}&
      16&
      3&
      {-9}\\
      14&
      {-5}&
      20&
      9&
      1\\
      4&
      1&
      {-4}&
      {-12}&
      {-24}\\
      \end{pmatrix},
\begin{pmatrix}{-43}&
      {-50}&
      {-27}&
      {-51}&
      9\\
      12&
      {-24}&
      36&
      0&
      {-12}\\
      35&
      34&
      27&
      39&
      {-9}\\
      {-3}&
      {-10}&
      3&
      {-6}&
      {-1}\\
      {-11}&
      {-10}&
      {-9}&
      {-12}&
      3\\
      \end{pmatrix},
\begin{pmatrix}{-8}&
      {-16}&
      {-12}\\
      {-5}&
      {-1}&
      {-15}\\
      {-1}&
      13&
      {-14}\\
      12&
      12&
      28\\
      {-1}&
      25&
      {-24}\\
      \end{pmatrix}.
$$}
The SVD normal form of the complex is given by the matrices
{\small
$$\begin{pmatrix}34.489&
      0&
      0&
      0&
      0\\
      0&
      28.714&
      0&
      0&
      0\\
      0&
      0&
      0&
      0&
      0\\
      \end{pmatrix},
\begin{pmatrix}0&
      0&
      0&
      0&
      0\\
      0&
      0&
      0&
      0&
      0\\
      114.08&
      0&
      0&
      0&
      0\\
      0&
      47.193&
      0&
      0&
      0\\
      0&
      0&
      0&
      0&
      0\\
      \end{pmatrix},\begin{pmatrix}0&
      0&
      0\\
      0&
      0&
      0\\
      45.993&
      0&
      0\\
      0&
      35.209&
      0\\
      0&
      0&
      0\\
      \end{pmatrix}.
$$}
So each of the matrices $A_i$ has rank $2$, and all homology groups $H_i$ are 1-dimensional.

The transformation into the normal form is given by the orthogonal matrices
{\small
$$
\begin{pmatrix}{-.6553}&
      .2393&
      {-.7165}\\
      {-.7549}&
      {-.1745}&
      .6322\\
      .0262&
      .9551&
      .2950\\
      \end{pmatrix},
\begin{pmatrix}{-.5694}&
      .1646&
      {-.7702}&
      {-.1318}&
      .1950\\
      .1862&
      .0303&
      .0679&
      {-.9710}&
      .1301\\
      {-.7448}&
      {-.1213}&
      .6010&
      {-.0706}&
      .2537\\
      {-.2631}&
      {-.4289}&
      {-.0790}&
      {-.1821}&
      {-.8411}\\
      .1309&
      {-.8794}&
      {-.1862}&
      .0404&
      .4162\\
      \end{pmatrix},$$
 $$
 \begin{pmatrix}.5019&
      {-.1770}&
      .2288&
      .5338&
      .6160\\
      .5257&
      .6126&
      .3335&
      .1127&
      {-.4738}\\
      .3586&
      {-.7250}&
      .3461&
      {-.3015}&
      {-.3677}\\
      .5735&
      .0970&
      {-.5972}&
      {-.5061}&
      .2210\\
      {-.1195}&
      .2417&
      .6000&
      {-.5961}&
      .4604\\
      \end{pmatrix},
 \begin{pmatrix}{-.2525}&
      {-.2843}&
      {-.9249}\\
      .1813&
      {-.9528}&
      .2434\\
      {-.9505}&
      {-.1062}&
      .2921\\
      \end{pmatrix}
$$}
which we have printed here with 4 valid digits only. In other words, the diagram
$$
\xymatrix{ 0 & \ar[l] \RR^3& \ar[l]_{A_1}  \RR^5 &\ar[l]_{A_2}  \RR^5 &\ar[l]_{A_3}  \RR^3 &\ar[l]  0 \cr
0 & \ar[l] \ar[u]^{U_0}  \RR^3& \ar[l]^{\overline \Sigma_1} \ar[u]^{U_1} \RR^5 &\ar[l]^{\overline \Sigma_2} \ar[u]_{U_2} \RR^5 &\ar[l]^{\overline \Sigma_3}  \ar[u]_{U_3} \RR^3 &\ar[l]  0 \cr
}
$$
commutes (up to the chosen precision).
The first matrix $A_1^+$ of the pseudoinverse complex over $\RR_{53}$ printed with 6 valid digits is
$$
 \begin{pmatrix}.0121907&
      .0114627&
      .0050431\\
      {-.00328525}&
      {-.00426002}&
      .00115014\\
      .013141&
      .0170401&
      {-.00460058}\\
      .00142545&
      .00836608&
      {-.0144655}\\
      {-.00981498}&
      .00248076&
      {-.0291523}\\
      \end{pmatrix}$$
which is an approximation of
 $$\begin{pmatrix}5978/490373&
      5621/490373&
      2473/490373\\
      {-1611/490373}&
      {-2089/490373}&
      564/490373\\
      6444/490373&
      8356/490373&
      {-2256/490373}\\
      699/490373&
      8205/980746&
      {-14187/980746}\\
      {-4813/490373}&
      2433/980746&
      {-28591/980746}\\
      \end{pmatrix}$$
\end{example}

\begin{remark}This simple example is pretty stable against errors. If we disturb the entries of the matrices in the complex arbitrarily by an relative
error of $\le 10^{-3}$, then taking $10^{-2}$ as a threshold the algorithms predicts the dimension of the homology groups still
correctly, see \href{http://www.math.uiuc.edu/Macaulay2/doc/Macaulay2-1.10/share/doc/Macaulay2/SVDComplexes/html}{SVDComplexes}.\end{remark}

\section{Projection}\label{applications}

One application of using the singular value
decomposition of a complex is to compute
the pseudoinverse complex as described
in Section~\ref{basics}.
In this section, we consider projecting
a sequence of matrices onto a complex.

\begin{algorithm}[Projection to a complex]\label{projection to a complex} $\quad$\\
INPUT: A sequence $B_1,\ldots,B_n$ of $c_{i-1} \times c_i$ matrices and a sequence
$h_0,\ldots,h_n$ of desired dimension of homology groups. \\
OUTPUT: A sequence $A_1,\ldots,A_n$ of matrices which define a complex with desired homology, if possible.\\
\begin{enumerate}
\item Set $r_0=0$ and compute $r_1,\ldots,r_{n+1}$ from
$ c_i =r_i+r_{i+1}+h_i$ recursively.
If $r_i <0 $ for some $i$ or $r_{n+1} \not=0$, then return the error message: ``The rank conditions cannot be satisfied.''
\item Set $Q_0=0$ and $P_0 = \id_{C_0}$.
\item For $i=1,\dots,n$
\begin{enumerate}[a.]
\item Compute the $(c_{i-1}-r_{i-1}) \times c_i$ matrix $\widetilde B_i= P_{i-1}\circ B_i$.
\item Compute the  singular value decomposition
$$ \widetilde B_i=\widetilde U_{i-1}\circ \widetilde \Sigma_i \circ \widetilde V_i^t.$$
\item Define
$$ \overline \Sigma_i = \bordermatrix{
& r_{i} & r_{i+1} & h_i \cr
r_{i-1} & 0 & 0& 0 \cr
r_{i} & \Sigma_i  & 0 & 0 \cr
h_{i-1} & 0 & 0 & 0 \cr
}
$$
as a block matrix where $\Sigma_i$ is a diagonal matrix with entries the largest~$r_i$ singular values of $\widetilde B_i$.

\item  Decompose
$$\widetilde V_i^t=\begin{pmatrix}Q_i  \cr  P_i \end{pmatrix}$$
 into submatrices consisting of the first $r_i$  and last $c_i-r_i$ rows of $\widetilde V_i^t$.
\item Define $$U^t_{i-1}=\begin{pmatrix}Q_{i-1}   \cr
\widetilde U^t_{i-1}\circ P_{i-1} \end{pmatrix}.$$
\item If $i=n$, then $U_n=\widetilde V_n^t$.
\end{enumerate}
\item Set $A_i = U_{i-1} \circ \overline \Sigma_i \circ U_i^t$  and
return $A_1,\ldots, A_n$.
\end{enumerate}
\end{algorithm}

\begin{remark}
By construction, it is clear that
Algorithm~\ref{projection to a complex}
computes a complex.  We leave it as an open problem
to compute the ``closest'' complex
to the given matrices $B_1,\dots,B_n$.
\end{remark}

\begin{example}
In our package \href{https://www.math.uni-sb.de/ag/schreyer/index.php/computeralgebra}{RandomComplexes}, we have implemented several methods to produce complexes over the integers. The first function randomChainComplex takes as input a sequences h and r of desired
dimension of homology groups and ranks of the matrices. It uses
the
 LLL algorithm~ \cite{LLL} to produce example of desired moderate height. It runs fast for complexes of ranks $c_i \le 100$ but  is slow for larger examples because of the use of the LLL-algorithm. Example \ref{exampleOfSection2} was produced this way.

We test Algorithms~\ref{SuccessiveProjectionMethod}
and~\ref{LaplacianMethod} to verify
the desired dimension of the homology groups.
Table~\ref{Tab:Comparison} compares the timings
of these two algorithms on various examples of this sort.

\begin{table}
\centering
\begin{tabular}{|l |l |c |c |}\hline
$c_0,\ldots c_3$ & $h_0,\ldots,h_3$ & Alg.~\ref{SuccessiveProjectionMethod} (sec) & Alg.~\ref{LaplacianMethod} (sec) \cr \hline
$7, 21, 28, 14$&$2, 3, 2, 1$& .00211& .0110\cr
$8, 27, 35, 17$&$3, 6, 4, 2$& .00225& .0182\cr
$9, 33, 42, 20$&$4, 9, 6, 3$&.00254&.0294\cr
$10, 39, 49, 23$&$5, 12, 8, 4$& .00291& .0647\cr
$11, 45, 56, 26$&$6,15, 10, 5$& .00355& .1090\cr
$12, 51, 63, 29$&$7, 18, 12, 6$&.00442&.1150\cr
\hline
\end{tabular}
\caption{Comparison of timings using Algorithms~\ref{SuccessiveProjectionMethod}
and~\ref{LaplacianMethod}.}\label{Tab:Comparison}
\end{table}
\end{example}

\begin{example}
Our second series of examples is constructed from Stanley-Reisner simplicial complexes of randomly chosen square free monomial ideals.
In the specific cases below, we selected $N$ square free monomials at random in a polynomial ring with $k$ variables which are summarized
in Table~\ref{Tab:Comparison2}.
Algorithm~\ref{LaplacianMethod}
does not apply to these examples since
repeated eigenvalues occur.

\begin{table}
{\small
\begin{tabular}{|l |l |rrrrrrrrrr |l  |}\hline
$k$&$N$&$c_0$&$c_1$&$c_2$&$\ldots$ & &&&&&&Alg~\ref{SuccessiveProjectionMethod}\cr
&&$h_0$&$h_1$&$h_2$&$\ldots$ & &&&&&&
\multicolumn{1}{c|}{(sec)}\cr\hline\hline
8& 20&8& 27& 44& 30&&&&&&& \cr
&&1& 0& 0& 1&&&&&&& .00185\cr\hline
9& 21& 9& 35& 74& 85& 46&&&&&& \cr
&&1& 0& 0& 0& 0&&&&&& .0036\cr\hline
10&23&10& 45& 118& 190& 173& 69&&&&&\cr
&&1& 0&0& 0& 3& 0&&&&& .0198\cr\hline
11& 26&11&55& 165& 326& 431& 361& 156& 19&&&\cr
&&1&0& 0& 0& 0& 0& 2& 0&&& .241\cr\hline
12& 30&12& 66& 218& 474& 694& 664& 375&101&&&\cr
&&1& 0& 0& 0& 0& 0& 2& 0&&& 1.29\cr\hline
13& 35&13& 78& 286& 712& 1253&1553& 1291& 639& 141&&\cr
&& 1& 0& 0& 0& 0& 0& 0& 6& 1&&39.7\cr\hline
14& 41&14&91&364&996&1948&2741&2687&677&559&75& \cr
&&1& 0& 0& 0& 0& 0& 0& 7& 0& 0& 355.\cr
\hline
\end{tabular}}
\caption{Comparison of timings using Algorithm~\ref{SuccessiveProjectionMethod}.
}\label{Tab:Comparison2}
\end{table}

\end{example}

\section{Application to syzygies}\label{syzygies}

We conclude with an application
concerning the computation of Betti numbers in free resolutions.
Let $S=K[x_0,\ldots,x_n]$ be the standard graded polynomial ring and $M$ a finitely generated graded $S$-module. Then,
by Hilbert's syzygy theorem,~$M$ has a finite free resolution:
$$
0 \lTo M \lTo F_0 \lTo^{\; \varphi_1} F_1 \lTo^{\; \varphi_2} \ldots \lTo^{\; \varphi_c} F_c \lTo 0
$$
by free graded $S$-modules $F_i =\sum_j S(-i-j)^{b_{ij}}$
of length $c \le n+1$. Here $S(-\ell)$ denotes the free $S$-module with generator in degree $\ell$.

If we choose in each step a minimal number of homogenous generators, i.e., if $\varphi_i(F_i) \subset (x_0,\ldots x_n) F_{i-1}$,
then the free resolution is unique up to
an isomorphism. In particular, the Betti numbers $b_{ij}$ of a minimal resolution are numerical
invariants of $M$.  On the other hand, for basic applications of free resolutions such as the computation of $\Ext$ and $\Tor$-groups, any resolution can be used.

Starting with a reduced Gr\"obner basis of the submodule $\varphi_1(F_1) \subset F_0$ there is, after some standard choices on orderings, a free resolution such that at each step the columns of $\varphi_{i+1}$ form a reduced Gr\"obner basis of $\ker \varphi_i$. This resolution is uniquely determined however, in most cases, highly nonminimal. An algorithm to compute this standard nonminimal resolution was developed in \cite{EMSS} which turned out to be much faster then the computation of a minimal resolution by previous methods.

The following forms the examples which we use as test cases.

\begin{proposition}[Graded Artinian Gorenstein Algebras]\label{AGR} Let $f \in \QQ[x_0,\ldots,x_n]$ be a homogeneous polynomial of degree $d$.
In $S=\QQ[\partial_0,\ldots,\partial_n]$,
consider the ideal $I=\langle D \in S \mid D(f)=0 \rangle$ of constant differential operators which annihilate~$f$.
Then, $A^{\perp}_f := S/I$ is an artinian Gorenstein Algebra with socle in degree~$d$.
\end{proposition}

For more information on this topic see, e.g., \cite{RS}.

\begin{example} Let
$f= \ell_1^4+\ldots+\ell_{18}^4 \in \QQ[x_0,\ldots,x_7]$
be the sum of $4^{\rm th}$ powers of 18 sufficiently general chosen linear forms $\ell_s$.
The Betti numbers $b_{ij}$ of the minimal resolution $M=A^\perp_f$ as an $S$-module are zero outside the range
${i=0,\ldots, 8, j=0,\ldots, 4}$. In this range, they take the values:
\begin{center}\begin{tabular}{c|ccccccccc}
$j \setminus i $     &0&1&2&3&4&5&6&7&8\\ \hline
	\text{0}&1&\text{.}&\text{.}&\text{.}&\text{.}&\text{.}&\text{.}&\text{.}&\text{.}\\
	\text{1}&\text{.}&18&42&\text{.}&\text{.}&\text{.}&\text{.}&\text{.}&\text{.}\\
	\text{2}&\text{.}&10&63&288&420&288&63&10&\text{.}\\
	\text{3}&\text{.}&\text{.}&\text{.}&\text{.}&\text{.}&\text{.}&42&18&\text{.}\\
	\text{4}&\text{.}&\text{.}&\text{.}&\text{.}&\text{.}&\text
       {.}&\text{.}&\text{.}&1\\\end{tabular}\end{center}
which, for example, says that $F_2= S(-3)^{42} \oplus S(-4)^{63}$.  We note that the symmetry of the table is a well-known consequence of the Gorenstein property.

On the other hand the Betti numbers of the uniquely determined nonminimal resolution are much larger:
\begin{center}\begin{tabular}{c|ccccccccc}
    $j \setminus i$  &0&1&2&3&4&5&6&7&8\\ \hline
 \text{0}&1&\text{.}&\text{.}&\text{.}&\text{.}&\text{.}&\text{.}&\text{.}&\text{.}\\
 \text{1}&\text{.}&18&55&75&54&20&3&\text{.}&\text{.}\\
 	\text{2}&\text{.}&23&145&390&580&515&273&80&10\\\text{3}&\text{.}&7&49&147&245
      &245&147&49&7\\\text{4}&\text{.}&1&7&21&35&35&21&7&1\\\end{tabular}\end{center}
To deduce from this resolution the Betti numbers of the minimal resolution, we can use the formula
$$b_{ij} = \dim \Tor^S_i(M,\QQ)_{i+j}.$$
For example, to deduce $b_{3,2}=288$, we have to show that the $5^{\rm th}$ constant strand of the nonminimal resolution
$$ 0 \lTo \QQ^1 \lTo \QQ^{49}  \lTo \QQ^{390} \lTo \QQ^{54} \lTo 0$$
has homology only in one position.
\end{example}

The matrices defining the differential in the nonminimal resolution have polynomial entries whose coefficients in $\QQ$ can have very large height
such that the computation of the homology of the strands becomes infeasible.
There are two options, how we can get information about  the minimal Betti numbers:
\begin{itemize}
\item Pick a prime number $p$ which does not divided any numerator of the normalized reduced Gr\"obner basis
and then reduce modulo $p$
yielding a module~$M(p)$ with the same Hilbert function as $M$. Moreover, for all but finitely many primes $p$, the Betti numbers of $M$ as
an $\QQ[x_0,\ldots,x_n]$-module and of~$M(p)$ as $\FF_p[x_0,\ldots,x_n]$-module coincide.

\item Pass from a normalized reduced  Gr\"obner basis of $\varphi_1(F_1) \subset F_0$ to a floating-point approximation of the Gr\"ober basis. Since in the algorithm for the  computation of  the uniquely determined nonminimal resolution \cite{EMSS}, the majority of ground field operations
are multiplications, we can hope that this computation is numerically stable and that the
singular value decompositions of the linear strands will detect the minimal Betti numbers correctly.
\end{itemize}

\begin{example}
We experimented with artinian graded Gorenstein algebras constructed from randomly chosen forms $f \in \QQ[x_0,\ldots,x_7]$ in 8 variables which were the sum of $n$  $4^{\rm th}$ powers of linear forms where $11 \leq n \leq 20$.  This experiment showed that roughly
$95\%$ the Betti table computed via floating-point
arithmetic coincided with one computed over a finite field.  The reason for this was that the current implementation uses only double precision floating-point
computations which caused difficulty in
detecting zero singular values correctly.
This would be improved following Remark~\ref{RemarkPrec}.
\end{example}

We now consider a series of examples
related to the famous Green's conjecture on canonical curves which was proved in a landmark paper \cite{Voi05} for generic curves.
In $S=\QQ[x_0,\ldots,x_a,y_0,\ldots,y_b]$,
consider the homogeneous ideal~$J_e$ generated by the $2\times 2$ minors of
 $$
 \begin{pmatrix}
 x_0 & x_1 & \ldots &x_{a-1} \cr
 x_1 & x_2 & \ldots & x_a
 \end{pmatrix} \text{ and }
  \begin{pmatrix}
 y_0 & y_1 & \ldots &y_{b-1} \cr
 y_1 & y_2 & \ldots & y_b
 \end{pmatrix}
 $$
together with the entries of the $(a-1) \times (b-1)$ matrix
 $$
 \begin{pmatrix}
 x_0 & x_1 & x_2 \cr
 x_1 & x_2 & x_3 \cr
\vdots & \vdots & \vdots\cr
 x_{a-2}&x_{a-1}& x_a\cr
 \end{pmatrix}
 \begin{pmatrix}
 0& 0 &  e_2\cr
 0 & -e_1 & 0 \cr
 1 & 0 & 0 \cr
 \end{pmatrix}
  \begin{pmatrix}
 y_0 & y_1 & \ldots &y_{b-2} \cr
 y_1 & y_2 & \ldots & y_{b-1} \cr
 y_2 & y_3 & \ldots & y_{b} \cr
 \end{pmatrix}
 $$
 for some parameters $e_1,e_2\in\QQ$.
 Then, by \cite{ES18}, $J_e$ is the homogeneous ideal of an arithmetically Gorenstein surface $X_e(a,b) \subset \PP^{a+b+1}$ with trivial canonical bundle. Moreover, the generators of $J_e$ form a Gr\"obner basis.
 To verify the generic Green's conjecture for curves
 of odd genus $g=2a+1$, it suffices to prove, for some values $e=(e_1,e_2) \in \QQ^2$, that $X_e(a,a)$ has a ``natural'' Betti table, i.e., for each $k$ there is at most one pair $(i,j)$ with $i+j=k$ and $b_{ij}(X_e(a,a))\not=0$. For special values of $e=(e_1,e_2)$, e.g., $e=(0,-1)$, it is known that the resolution is not natural, see \cite{ES18}.

\begin{example}
 For $a=b=6$, our implementation computes the
 following Betti numbers for the nonminimal resolution:
 as
 \setcounter{MaxMatrixCols}{13}
 \begin{small}
$$
\begin{matrix}
      | &0&1&2&3&4&5&6&7&8&9&10&11\\ \hline
|&1&\text{.}&\text{.}&\text{.}&\text{.}&\text{.}&\text{.}&\text{.}&\text{.}&\text{.}&\text{.}&\text{.}\\
|&\text{.}&55&320&930&1688&2060&1728&987&368&81&8&\text{.}\\
|&\text{.}&\text{.}&39&280&906&1736&2170&1832&1042&384&83&8\\
|&\text{.}&\text{.}&\text{.}&1&8&28&56&70&56&28&8&1\\
\end{matrix}
$$
\end{small}

For $e=(2,-1)$ and $e=(0,-1)$,
our implementation correctly computes
the following Betti numbers,
respectively, of the minimal resolutions:
 \begin{small}
$$\begin{matrix}
      |&0&1&2&3&4&5&6&7&8&9&10&11\\ \hline
|&1&\text{.}&\text{.}&\text{.}&\text{.}&\text{.}&\text{.}&\text{.}&\text{.}&\text{.}&\text{.}&\text{.}\\
|&\text{.}&55&320&891&1408&1155&\text{.}&\text{.}&\text{.}&\text{.}&\text{.}&\text{.}\\
|&\text{.}&\text{.}&\text{.}&\text{.}&\text{.}&\text{.}&1155&1408&891&320&55&\text{.}\\
|&\text{.}&\text{.}&\text{.}&\text{.}&\text{.}&\text{.}&\text{.}&\text{.}&\text{.}&\text{.}&\text{.}&1\\
\end{matrix}$$
\end{small}
\begin{small}
$$\begin{matrix}\hline
|&1&\text{.}&\text{.}&\text{.}&\text{.}&\text{.}&\text{.}&\text{.}&\text{.}&\text{.}&\text{.}&\text{.}\\
|&\text{.}&55&320&900&1488&1470&720&315&80&9&\text{.}&\text{.}\\
|&\text{.}&\text{.}&9&80&315&720&1470&1488&900&320&55&\text{.}\\
|&\text{.}&\text{.}&\text{.}&\text{.}&\text{.}&\text{.}&\text{.}&\text{.}&\text{.}&\text{.}&\text{.}&1\\
\end{matrix}$$
\end{small}

Each of these computations took several minutes.
To consider larger examples, more
efficient algorithms and/or implementations
for computing the singular value decomposition of a complex are needed.
\end{example}


\vskip 0.5in

\vbox{\noindent Author Addresses:\par
\smallskip
\noindent{Danielle A. Brake}\par
\noindent{Department of Mathematics, University of Wisconsin - Eau Claire, Eau Claire WI 54702}\par
\noindent{brakeda@uwec.edu}\par
\smallskip
\noindent{Jonathan D. Hauenstein}\par
\noindent{Department of Applied and Computational Mathematics and Statistics, University of Notre Dame, Notre Dame IN 46556}\par
\noindent{hauenstein@nd.edu}\par
\smallskip
\noindent{Frank-Olaf Schreyer} \par
\noindent{Mathematik und Informatik, Universit\"at des Saarlandes, Campus E2 4, D-66123 Saarbr\"ucken, Germany}\par
\noindent{schreyer@math.uni-sb.de}\par
\smallskip
\noindent{Andrew J. Sommese} \par
\noindent{Department of Applied and Computational Mathematics and Statistics, University of Notre Dame, Notre Dame IN 46556}\par
\noindent{sommese@nd.edu}\par
\smallskip
\noindent{Michael E. Stillman}\par
\noindent{Department of Mathematics, Cornell University, Ithaca NY 14853}\par
\noindent{mike@math.cornell.edu}\par
}

\end{document}

%% file: preamble.tex
\def\antiddot{\mathinner{\mkern1mu\raise1pt\vbox{\kern7pt\hbox{.}}\mkern2mu
        \raise4pt\hbox{.}\mkern2mu\raise7pt\hbox{.}\mkern1mu}}

\newcommand{\FF}{{\mathbb F}}

\newcommand{\PP}{{\mathbb P}}
\newcommand{\QQ}{{\mathbb Q}}
\newcommand{\RR}{{\mathbb R}}

\newcommand{\Ext}{{\rm{Ext}}}

\newcommand{\id}{{\rm{id}}}

\newcommand{\punkt}{\hspace{-.3ex}\raise.15ex\hbox to1ex{\Huge.}}

\DeclareMathOperator{\image}{image}

\DeclareMathOperator{\Tor}{Tor}

\DeclareMathOperator{\rank}{rank}

\newtheorem{theorem}{Theorem}[section]

\newtheorem{proposition}[theorem]{Proposition}
\newtheorem{corollary}[theorem]{Corollary}

\theoremstyle{definition}

\newtheorem{remark}[theorem]{Remark}

\newtheorem{example}[theorem]{Example}

\newtheorem{algorithm}[theorem]{Algorithm}